\RequirePackage{rotating}

\documentclass{amsart}  

\usepackage[foot]{amsaddr}

\usepackage[usenames, dvipsnames]{color}
\usepackage{amsfonts}
\usepackage{amsthm}
\usepackage{amsmath}
\usepackage{amsfonts}
\usepackage{latexsym}
\usepackage{amssymb}
\usepackage{amscd}
\usepackage[latin1]{inputenc}
\usepackage{verbatim}
\usepackage{enumerate}
\usepackage{enumitem}
\usepackage{graphicx}
\usepackage{xcolor}
\usepackage{tikz}
\usepackage{caption}
\usepackage{adjustbox}
\usepackage{algorithmic}
\usepackage{epstopdf}
\usepackage{fancybox}
\usepackage{booktabs} 
\usepackage{expl3}
\usepackage{l3keys2e}
\usepackage{xparse}
\usepackage{blkarray}
\usepackage{pst-node}
\usepackage{animate}
\usepackage{float}
\usepackage{booktabs}
\usepackage{makecell}
\usepackage{systeme}
\usepackage{placeins}
\usepackage{multirow}

\usetikzlibrary{matrix}
\usetikzlibrary {positioning}
\usetikzlibrary{arrows,shapes}
\usetikzlibrary{trees}
\usetikzlibrary{shapes.geometric}
\usetikzlibrary{calc,shapes.callouts,shapes.arrows}
\usetikzlibrary{graphs}
\usetikzlibrary{positioning, arrows}

\definecolor{mybluegreen}{rgb}{0.1, 0.55, 0.35}
\usepackage{caption}
\usepackage{array,booktabs}
\usepackage{siunitx}
\usepackage{rotating}

\usepackage[spaces,hyphens]{url}
\usepackage[colorlinks,allcolors=blue]{hyperref}
\usepackage[math]{cellspace}
\newtheorem{theorem}{Theorem}[section]
\newtheorem{lemma}[theorem]{Lemma}

\newtheorem{cor}[theorem]{Corollary}

\newtheorem{conj}{Conjecture}[section]

\newtheorem{corollary}[theorem]{Corollary}
\theoremstyle{definition}

\DeclareMathOperator{\one}{\bf 1}

\newcommand{\Sym}{\mathrm{Sym}} 
\newcommand{\sym}[1]{\Sym({#1})}

\newcommand{\Alt}{\mathrm{Alt}} 
\newcommand{\alt}[1]{\Alt({#1})}

\DeclareMathOperator\Ind{ind}
\DeclareMathOperator\Res{res}

\newcommand\indg[2]{{\Ind{ ({#1}) }^{#2}  }}
\newcommand\resg[2]{{\Res{ ({#1}) }_{#2} }}

\definecolor{mybluegreen}{rgb}{0.1, 0.55, 0.35}
\newcolumntype{?}{!{\vrule width 2pt}}
\newcolumntype{M}[1]{>{\centering\arraybackslash}m{#1}}
\newcolumntype{P}[1]{>{\centering\arraybackslash}p{#1}}
\newcolumntype{N}{@{}m{0pt}@{}}
\newcommand{\cardinality}[1]{|{#1}|}
\title[EKR-theorem for uniform set partitions]{An Extension of the Erd\H{o}s-Ko-Rado Theorem to uniform set partitions}

\author[K.~Meagher]{Karen Meagher${^2}$ ${^*}$} 
\email[K.~Meagher]{karen.meagher@uregina.ca}
\address[K.~Meagher]{Department of Mathematics and Statistics, University of Regina, Regina, SK, S4S 0A2, Canada}
\author[M.~N.~Shirazi]{Mahsa N. Shirazi}
 \email[M.~N.~Shirazi]{mahsa.nasrollahi@gmail.com}
\address[M.~N.~Shirazi]{Department of Mathematics and Statistics, University of Regina, Regina, SK, S4S 0A2, Canada}
\author[B.~Stevens]{Brett Stevens${^1}$} \email[B.~Stevens]{brett@math.carleton.ca}
\address[B.~Stevens]{School of Mathematics and Statistics, Carleton University, Ottawa, ON, K1S 5B6, Canada}

\date{\today}

 		
\begin{document}

\begin{abstract}  
A $(k,\ell)$-partition is a set partition which has $\ell$ blocks each of size $k$. Two uniform set partitions $P$ and $Q$ are said to be partially $t$-intersecting if there exist blocks $P_{i}$ in $P$ and $Q_{j}$ in $Q$ such that $\left| P_{i} \cap Q_{j} \right|\geq t$. In this paper we prove a version of the Erd\H{o}s-Ko-Rado theorem for partially $2$-intersecting $(k,\ell)$-partitions. In particular, we show for $\ell$ sufficiently large, the set of all $(k,\ell)$-partitions in which a block contains a fixed pair is the largest set of 2-partially intersecting $(k,\ell)$-partitions. For for $k=3$, we show this result holds for all $\ell$.
\end{abstract}

\maketitle


\section{Introduction}
\label{sec:Intro}
In 1961, Erd\H{o}s, Ko, and Rado proved that if $\mathcal{F}$ is a $t$-intersecting family of $k$-subsets of $\{1,2,\ldots, n\}$, then $\binom{n-t}{k-t}$ is a tight upper bound on the size of $\mathcal{F}$, provided that $n$ is sufficiently large~\cite{MR0140419}. This result has motivated consideration of ``intersecting'' families of many other combinatorial objects using diverse proof techniques and has developed into an active and broad area of research. There are many recent results giving analogs of the EKR theorem; see, for example,~\cite{2int, MR3646689, MR3679841, MR2156694, MR0771733} or~\cite{MR3497070} and the references within. In this work, we prove an extension of the EKR theorem to systems of uniform set partitions.

A \textsl{$(k,\ell)$-partition} is a set partition of $\{ 1,2,\dots, k\ell \}$ with exactly $\ell$ blocks each of size $k$.
These are also called \textsl{uniform set partitions}. We use $\mathcal{U}_{k, \ell}$ to denote the set of all $(k,\ell)$-partitions, and $u_{k,\ell} = \left|  \mathcal{U}_{k, \ell} \right|$. It is easy to see that
\begin{equation}\label{eq:ukl}
u_{k,\ell} = \frac{1}{\ell!}\binom{k\ell}{k}\binom{k\ell-k}{k} \binom{k\ell-2k}{k} \cdots \binom{k}{k}.
\end{equation}

In~\cite{MR1752953}, Erd\H{o}s and Sz\'{e}kely considered different types of intersection for partitions. One of these types, and the one we consider here, two partitions $P$ and $Q$ are \textsl{intersecting in a pair} if there exist blocks $P_{i}$ in $P$, and $Q_{j}$ in $Q$ such that $\left| P_{i} \cap Q_{j} \right|\geq 2$.  Their work considers all partitions, not just uniform partitions.
In~\cite{MR2156694}, Meagher and Moura generalized this definition: two partitions $P$ and $Q$ are \textsl{partially $t$-intersecting}
if there exist $P_{i}$ in $P$, and $Q_{j}$ in $Q$ such that $\left| P_{i} \cap Q_{j} \right|\geq t$. 
This work is different than that of Erd\H{o}s and Sz\'{e}kely since only uniform partitions are considered in~\cite{MR2156694}.

A set of partitions is a \textsl{partially $t$-intersecting set} if any two partitions in the set are partially $t$-intersecting.
Meagher and Moura~\cite{MR2156694} conjectured that for $t\leq k$, if $\mathcal{P} \subset \mathcal{U}_{k, \ell}$ is a set of
partially $t$-intersecting partitions, then $\left| \mathcal{P} \right| \leq \binom{k\ell-t}{k-t}u_{k,\ell-1}.$  A set of this size can be formed by fixing a $t$-subset $T$ and the taking all $(k,\ell)$-partitions that have a block containing $T$; 
such a set is called a set of \textsl{canonically $t$-intersecting $(k,\ell)$-partitions}. Moreover, Meagher and Moura conjectured that only the canonically $t$-intersecting $(k,\ell)$-partitions have this maximum size.
As pointed out by Brunk in~\cite{brunkThesis}, this conjecture additionally requires that $k \leq  \ell(t-1)$, since if $k > \ell(t-1)$, then any two $(k,\ell)$-partitions are $t$-partially intersecting. 

If $k=t=2$, the the $(2,\ell)$-partitions are perfect matchings in the complete graph on $2\ell$ vertices. In this case, partially 2-intersecting is equivalent to intersecting (as sets). The Meagher-Moura conjecture has been proven in this case in~\cite{MR3646689}. In this paper we only consider $k\geq 3$.

In this work we prove the Meagher-Moura conjecture for $t=2$ with $k=3$ and all values of $\ell$, and for all $k\geq 4$, provided that $\ell$ is sufficiently large. Our approach is to define a graph in which the cocliques (also known as independent sets) are equivalent to partially $2$-intersecting $(k,\ell)$-partitions from $\mathcal{U}_{k,\ell}$. Then we use algebraic methods to find the size of a maximum coclique in the graph.

\section{Overview of Method}

Let $X$ be a graph. A \textsl{clique} in $X$ is a set of vertices for which their induced subgraph in $X$ is complete; and a 
\textsl{coclique} is a set of vertices in $X$ for which their induced subgraph is empty. The size of a largest clique and a 
largest coclique are denoted by $\omega(X)$ and $\alpha(X)$, respectively. The \textsl{adjacency matrix} $A(X)$ of $X$ is a 
matrix in which rows and columns are indexed by the vertices in $X$ and the $(i ,j)$-entry is 1 if $i$ and $j$ are adjacent, 
and 0 otherwise.  The \textsl{eigenvalues} of $X$ refer to the eigenvalues of its adjacency matrix.
We use $\mathbf{1}$ to denote the all-ones vector; for any $d$-regular graph, the all-ones vector is an eigenvector with eigenvalue $d$.
 
In general, finding the largest coclique of a graph $X$ is known to be NP-hard, 
but the \textsl{Delsarte-Hoffman (ratio)} bound gives an upper bound on $\alpha(X)$. This bound is based on the ratio 
between the largest and the smallest eigenvalue of the adjacency matrix of the graph. A proof of this result can be found in~\cite[Section 2.4]{MR3497070}.

\begin{theorem}[Delsarte-Hoffman bound]~\cite{MR0384310}\label{Thm:RatioBound}
Let $A$ be the adjacency matrix for a $d$-regular graph $X$ on vertex set $V(X)$. If the least eigenvalue of $A$ is $\tau$, then 
\begin{equation*}\label{eq:RatioBound}
\alpha(X)\leq \frac{|V(X)|}{1-\frac{d}{\tau}}.
\end{equation*}
If equality holds for some coclique $S$ with characteristic vector $\nu_{S}$, then 
\begin{equation*}
\nu_{S}-\frac{|S|}{|V(X)|}\mathbf{1}
\end{equation*}
is an eigenvector with eigenvalue $\tau$.
\end{theorem}


Define $X_{k,\ell}$ to be the graph with $\mathcal{U}_{k, \ell}$ as its vertex set, in which two partitions 
$P$ and $Q$ are adjacent if every pair of blocks, one from $P$ and one from $Q$, have at most $1$ element in common. 
The group $\sym{k\ell}$ acts transitively on the vertices of $X_{k, \ell}$ and preserves the edges. This means the $X_{k,\ell}$
 is vertex transitive and regular. We will denote the degree by $d_{k,\ell}$, or simply $d$ when the context is clear.

A resolvable packing design on $k\ell$ points with blocksize $k$ and index $\lambda=1$ is equivalent to a clique in this graph. Further, a resolvable balanced incomplete block design on $k\ell$ points with blocksize $k$ and index $\lambda=1$, if it exists, gives a maximum clique.  

For any distinct $i, j \in \{1,\dots, k\ell\}$, let $S_{i, j}$ be the subset of partitions in $\mathcal{U}_{k, \ell}$ for which the elements $i$ and $j$ are in the same block. Then $S_{i, j}$ is a coclique in the graph $X_{k, \ell}$ and the size of $S_{i,j}$ is 
\begin{equation*}
 \frac{1}{(\ell-1)!}\binom{k\ell-2}{k-2}\binom{k\ell-k}{k}\cdots \binom{k}{k}.
\end{equation*}
The main goal in this paper is to prove, using the ratio bound, that $S_{i, j}$ is a maximum coclique in $X_{k, \ell}$. 
For the ratio bound to hold with equality, we need to prove if $\tau$ is the least eigenvalue of $X_{k,\ell}$, then
\begin{equation*}
1-\frac{d_{k,\ell} }{\tau} =  \frac{u_{k,\ell}}{|S_{i,j}|} = \frac{k\ell-1}{k-1}.
\end{equation*}
Thus we need to prove two facts: first that $ \tau = - \frac{ d_{k,\ell} (k-1) }{ k(\ell- 1) }$ is an eigenvalue of $X_{k,\ell}$; and second that $\tau$ is the least eigenvalue of $X_{k,\ell}$.

In the next section, we show how the eigenvalues of $X_{k,\ell}$ are connected to the representations of $\sym{k\ell}$, and we prove some bounds on the degrees of the irreducible representations of $\sym{k\ell}$. Next, in Section~\ref{sec:evals}, we calculate three of the eigenvalues of $X_{k,\ell}$; one of these eigenvalues is the $\tau$ above. To prove that $\tau$ is the least eigenvalue, in Section~\ref{sec:multiplicity}, we show if there is another eigenvalue, strictly smaller than $\tau$, then its multiplicity must be bounded by a function that includes the ratio $u_{k,\ell}/d_{k,\ell}$.
In Section~\ref{sec:degree}, we show that the limit of ratio $u_{k,\ell}/d_{k,\ell}$ is finite as $\ell \rightarrow \infty$.  This gives a simple upper bound on $u_{k,\ell}/d_{k,\ell}$ for all sufficiently large $\ell$.  In Section~\ref{sec:multiplicity} we use the bounds from Section~\ref{sec:representations} so show that no such eigenvalues exist. This proves the Meagher-Moura Conjecture with $t=2$, for all values of $k$, provided that $\ell$ is sufficiently large. Finally, in Section~\ref{sec:exact3}, we find a prove a weaker bound for $u/d$ when $k=3$ but one that holds for all $\ell$.  Thus we prove the Meagher-Moura Conjecture for $t=2$, $k=3$ for all values of $\ell$.


\section{Representations of the Symmetric Group}\label{sec:representations}

In this section we will explain the connection between the eigenvalues of the graph $X_{k,\ell}$ and the 
irreducible representations of the symmetric group. We also give results on the dimensions of the 
irreducible representations that are involved in the eigenvalues.

For any character $\chi$ of $\sym{n}$, we can consider its restriction to $H \leq \sym{n}$ which is denoted by $\resg{\chi}{H}$. Similarly if $\chi$ is a representation of $H  \leq \sym{n}$, then its induced representation on $\sym{n}$ is denoted by $\indg{\chi}{\sym{n}}$. The trivial character on a group $H$ is denoted by $1_H$.

The stabilizer of a partition in $\mathcal{U}_{k, \ell}$ is the group $\sym{k} \wr \sym{\ell}$ (this is called the \textit{wreath product} of $\sym{k}$ and $\sym{\ell}$). The cosets $\sym{k\ell}/(\sym{k} \wr \sym{\ell})$ are in one-to-one correspondence with the partitions of $\mathcal{U}_{k, \ell}$.  The action of $\sym{k\ell}$ on the partitions is equivalent to the action of $\sym{k\ell}$ on the cosets $\sym{k\ell}/(\sym{k} \wr \sym{\ell})$ and this action is clearly transitive. The permutation representation of this action is
\[
\indg{1_{\sym{k} \wr \sym{\ell}}}{ \sym{k\ell} }.
\] 
The module for this representation can be thought of as the vector space 
of length-$u_{k,\ell}$ vectors with the characteristic vectors of  $P \in \mathcal{U}_{k,\ell}$, denoted by $v_P$, as its basis. 
The group $\sym{k\ell}$ acts on this vector space by the action on the partitions, for any $\sigma \in \sym{k\ell}$ the action is $\sigma(v_P) = v_{P^\sigma}$.

This representation can be decomposed as the sum of irreducible representations of $\sym{k\ell}$. If the multiplicity of each irreducible representation in the decomposition is equal to 1, then the representation is called \textsl{multiplicity-free}. In general, the group $\sym{k} \wr \sym{\ell}$ is not multiplicity free in $\sym{k\ell}$. In fact it is not multiplicity free unless $k=2$, $\ell =2$, or $(k,\ell)$ is one of (3,3), (4,3), (5,3) or (3,4)~\cite{MR2581098}.  

\subsection{Orbital Association Scheme}

The \textsl{orbitals} of the action of a group $G$ on a set $\Omega$ is the set of orbits of the action of $G$ on 
$\Omega \times \Omega$; so the orbitals are the orbits of the action of $G$ on the pairs from $\Omega$. 
Each orbital of $\sym{k\ell}$ on $\sym{k\ell}/(\sym{k} \wr \sym{\ell})$ can be represented by an object called a \textsl{meet table}. The meet table for two $(k,\ell)$-partitions is a $\ell \times \ell$ array in which the $(i, j)$-entry is $|P_i \cap Q_j |$. Two meet tables are \textsl{isomorphic} if one can be obtained from the other by permuting the rows and the columns. In~\cite[Section 15.4]{MR3497070} it is shown that the set of non-isomorphic meet tables correspond to the set of orbitals. For each orbital $\mathcal{O}$ there is a corresponding meet table $M$; this means for $P,Q \in \mathcal{U}_{k, \ell}$ the meet table of $P$ and $Q$ is $M$ if and only if $(P,Q) \in \mathcal{O}$. 
Further, each orbital can be represented as a $u_{k,\ell} \times u_{k,\ell}$ matrix, with the $(P,Q)$-entry equal to 1 if and only if the meet table of $P$ and $Q$ is isomorphic to the table representing the orbital. 
The set of these $u_{k,\ell} \times u_{k,\ell}$-matrices of the orbitals forms an \textsl{association scheme} if and only if $\indg{1_{\sym{k} \wr \sym{\ell}}}{ \sym{k\ell} }$ is multiplicity-free. In general, these matrices form a \textsl{homogeneous coherent configuration}.
 
The graph $X_{k,\ell}$ is the union of the orbitals from the action of $\sym{k\ell}$ on $\sym{k\ell}/(\sym{k} \wr \sym{\ell})$ that are represented by a meet table that has no entry greater than 1. This means that $X_{k,\ell}$ is in the commutant of the permutation representation of $\sym{k\ell}$. In particular, for any $\sigma \in \sym{k\ell}$, with permutation representation $M_\sigma$ we have that
\[
M_{\sigma^{-1}} A(X_{k,\ell}) M_{\sigma} =  A(X_{k,\ell}). 
\]
Further, if $v$ is any $\theta$-eigenvector of $X_{k,\ell}$, then $M_{\sigma}v$ is also a $\theta$-eigenvector.
This implies the eigenspaces of $X_{k,\ell}$ are invariant under the action of $\sym{k\ell}$ and thus a 
union of irreducible modules in the decomposition of  
\[
\indg{1_{\sym{k} \wr \sym{\ell}}}{ \sym{k\ell} }.
\] 
We say that an eigenvalue $\theta$ \textsl{belongs} to a module if the module is a subspace of the $\theta$-eigenspace.

\subsection{Dimensions of the representations of $\sym{k\ell}$}

In this section we will give some results on the irreducible representations of $\sym{n}$. We refer the reader to~\cite{MR1824028}, or any similar reference on this topic, for details and background.
It is well-known that the irreducible representations of $\sym{n}$ correspond to integer partitions on 
$n$.  We will use $\lambda \vdash n$ to 
indicate that $\lambda$ is an integer partition of $n$, this means that $\lambda = [\lambda_1,\lambda_2,\dots, \lambda_j]$, 
each $\lambda_i$ is an integer and $\sum_{i=1}^j \lambda _i = n$.
We will use $\chi_\lambda$ to represent the irreducible character of $\sym{n}$ corresponding to the partition 
$\lambda$. 

From \cite{MR3646689} we have a list of irreducible representations of the symmetric group with small degree

\begin{lemma}\label{lem:eightspecial}
  For $n \geq 9$, let $\chi$ be a representation of $\sym{n}$ with
  degree less than $(n^2-n)/2$. If $\chi_\lambda$ is a constituent of
  $\chi$, then $\lambda$ is one of the following partitions of $n$:
\[
[n],\, [1^n], \, [n-1,1],\, [2,1^{n-2}],\, [n-2,2], \, [2,2,1^{n-4}],\, [n-2,1,1], \, [3,1^{n-3}].
\]
\end{lemma}

This proof uses the branching rule, which we state here. 
For a proof of this rule see~\cite[Corollary 3.3.11]{MR2643487}.

\begin{lemma}\label{lem:branching}
Let $\lambda \vdash n$, then
\[
\resg{\chi_\lambda}{\sym{n-1}} = \sum \chi_{\lambda^-},
\]
where the sum is taken over all partitions $\lambda^{-}$ of $n-1$ that have a
Young diagram which can be obtained
by the deletion of a single box from the Young diagram of $\lambda$.
Further,
\[
\indg{\chi_\lambda}{\sym{n+1}} = \sum \chi_{\lambda^+},
\]
where the sum is taken over partitions $\lambda^+$ of $n+1$ that have a
Young diagram which can be obtained by the addition of a single box to
Young diagram of $\lambda$. \qed
\end{lemma}
 
Using the same approach as the proof for Lemma~\ref{lem:eightspecial} we can get a second 
family of representations with slightly larger, but still small dimension.

\begin{lemma}\label{lem:tenspecial}
  For $n \geq 13$, let $\chi$ be an irreducible representation of $\sym{n}$ with
  degree less than $\binom{n}{3} - \binom{n}{2}$. If $\chi_\lambda$ is a constituent of
  $\chi$, then $\lambda$ is one of the following partitions of $n$:
\begin{align*}
 & [n],  \quad [1^n],  \quad [n-1,1], \quad [2,1^{n-2}],  \quad   [n-2,2],  \quad  [2,2,1^{n-4}], & \\
 & [n-2,1,1],    \quad  [3,1^{n-3}], \quad [n-3,3],  \quad  [2,2,2,1^{n-6}]. &
\end{align*}
\end{lemma}

\proof  
The hook length formula confirms that each of the 10 representations above 
have degree less than or equal to  $\binom{n}{3} - \binom{n}{2}$.

We prove this result by induction. For $n=13$ and $14$ this can be calculated directly
using the GAP character table library~\cite{GAP4}. 
We assume for $n\geq 14$ that the lemma holds for $n$ and
$n-1$, and we will prove that the lemma holds for $n+1$.

Assume that $\chi$ is an irreducible representation of $\sym{n+1}$ that has
dimension less than 
\[
\binom{n+1}{3} - \binom{n+1}{2} = \frac{(n+1)n(n-4)}{6},
\]
but is not one of the ten irreducible representations listed in the statement of the lemma. 
We will show that such a $\chi$ cannot exist.

If one of the ten irreducible representations of $\sym{n}$ with dimension less than
$\binom{n}{3} - \binom{n}{2}$ is a constituent of $\resg{\chi}{\sym{n}}$, then we can determine 
the possible constituents of $\chi$ with the branching rule. 

\renewcommand{\arraystretch}{1.25}
\begin{table}[h!]
\begin{center}
\begin{tabular}{|l|l|} \hline
Constituent of $\resg{\chi}{\sym{n}}$  & Constituents of $\chi$ \\ \hline
$[n]$ & $[n+1]$, $[n,1]$\\   
$[n-1,1]$ & $[n,1]$, $[n-1,2]$, $[n-1,1,1]$ \\ 
$[n-2,2]$ & $[n-1,2]$, $[n-2,3]$, $[n-2,2,1]$ \\  
$[n-2,1,1]$ & $[n-1,1,1]$, $[n-2,2,1]$, $[n-2,1,1,1]$ \\ 
$[n-3, 3]$ & $[n-2,3]$, $[n-3,4]$, $[n-3,3,1]$ \\ 
$[1^n]$ & $[2,1^{n-1}]$, $[1^{n+1}]$ \\  
$[2,1^{n-2}]$ & $[3,1^{n-2}]$, $[2,2,1^{n-3}]$, $[2,1^{n-1}]$ \\ 
$[2,2, 1^{n-4}]$ & $[3,2,1^{n-4}]$, $[2,2,2,1^{n-5}]$, $[2,2,1^{n-3}]$ \\ 
$[3, 1^{n-3}]$ & $[4,1^{n-3}]$, $[3,2,1^{n-4}]$, $[3,1^{n-2}]$ \\  
$[2,2,2, 1^{n-6}]$ & $[3,2,2,1^{n-6}]$, $[2,2,2,2,1^{n-7}]$, $[2,2,2,1^{n-5}]$ \\ \hline 
\end{tabular}
\caption{Constituents of $\chi$, if $\resg{\chi}{\sym{n}}$ has a constituent with degree less than $\binom{n}{3} - \binom{n}{2}$.}\label{tab:tableofsmallreps}
\end{center}
\end{table}
\renewcommand{\arraystretch}{1}

\renewcommand{\arraystretch}{1.25}
\begin{table}[h!]
\begin{center}
\begin{tabular}{|l|l|} \hline
Representation & Degree \\ \hline
$[n-3,4]$ & $(n+1)n(n-1)(n-7)/24$ \\
$[n-3,3,1]$ & $(n+1)n(n-2)(n-5)/8$ \\
$[n-2,2,1] $& $(n+1)(n-1)(n-3)/3$\\
$[n-2,1,1,1] $& $n(n-1)(n-2)/6$\\
$[2,2,2,2,1^{n-8}]$ & $(n+1)n(n-1)(n-7)/24$ \\
$[3,2,2,1^{n-6}]$ & $(n+1)n(n-2)(n-5)/8$\\
$[3,2,1^{n-4}]$&  $(n+1)(n-1)(n-3)/3$ \\
$[4,1^{n-3}] $& $n(n-1)(n-2)/6$ \\ \hline
\end{tabular}
\caption{Degrees of the representations from Table~\ref{tab:tableofsmallreps} that
are larger than $ \frac{(n+1)n(n-4)}{6}$ for $n \geq 13$.}\label{tab:mediumreps}
\end{center}
\end{table}
\renewcommand{\arraystretch}{1}

By Frobenius reciprocity,  for any representation $\phi$ of $\sym{n}$ 
\[
\langle \resg{\chi}{\sym{n}} , \, \phi \rangle_{\sym{n}} =
\langle  \chi , \, \indg{\phi}{\sym{n+1}}  \rangle_{\sym{n+1} }.
\]
This means if $\phi$ is a constituent of $\resg{\chi}{\sym{n}}$, then $\chi$ is one of the constituents of $\indg{\phi}{\sym{n+1}}$. 
The possible constituents of $\indg{\phi}{\sym{n+1}}$ are recorded in Table~\ref{tab:tableofsmallreps}; 
the second column lists the irreducible representations that, according to the branching rule, are constituents of 
representation of $\sym{n+1}$ induced by the representation in the first column.

From these lists, and the degrees of the representations given in Table~\ref{tab:mediumreps}, we see that either $\chi$ is one of  the ten listed in the theorem, or the dimension of $\chi$ is larger than $\binom{n}{3} - \binom{n}{2}$  (again, the dimensions are calculated using the hook length formula). 
Thus $\resg{\chi}{\sym{n}}$ does not contain any of the ten irreducible representations of $\sym{n}$ in the statement of the theorem.

Next consider the case where the decomposition of $\resg{\chi}{\sym{n}}$ contains at least two irreducible
representations of $\sym{n}$ which are not in the list of the ten
irreducible representations with dimension less 
$\binom{n}{3} - \binom{n}{2} = n(n-1)(n-5)/6$. In this case, the
dimension of $\chi$ must be at least $n(n-1)(n-5)/3$.
But since $n>7$, this is strictly larger than $(n+1)n(n-4)/6$. 

Finally we need to consider the case where $\resg{\chi}{\sym{n}}$ contains exactly one irreducible representation of $\sym{n}$,
which is not one of the ten listed in the theorem. By the branching rule 
the only irreducible representations of $\sym{n+1}$ for which
 $\resg{\chi}{\sym{n}}$ contains only one irreducible representation have a rectangular Young diagram, so 
$\chi = \chi_{[s^t]}$ for some $s$ and $t$.

Next consider $\resg{\chi}{\sym{n-1}}$, this is the restriction of $\chi = \chi_{[s^t]}$ to $\sym{n-1}$.
By the branching rule, this can contain only the irreducible representations
of $n-1$ that correspond to the partitions $\lambda' = [s^{t-1}, s-2]$ and
$\lambda'' = [s^{t-2}, s-1, s-1]$. 

If $\lambda'$ is one of the ten partitions that correspond to irreducible
representations of $\sym{n-1}$ dimension less than
$\binom{n-1}{3} - \binom{n-1}{2}$, then one of the following cases must hold
\begin{itemize}
\item $t=1$ and $\lambda' = [n-1]$ and $s=n+1$;
\item $t=2$ and $\lambda' = [n-1,1], [n-2,2]$ or $[n-3,3]$, and $s \leq 5$; or
\item $2<t<4$ and and $s\leq 2$.
\end{itemize}
The first of these cases implies $\chi = [n+1]$, which contradicts the dimension of $\chi$, and 
none of the other cases can happen, since $n = st$ and $n$ is
assumed to be at least $13$.

Similarly, assume $\lambda''= [s^{t-2}, s-1, s-1]$ is one of the partitions corresponding to the ten representations of
$\sym{n-1}$ that have dimension less than $\binom{n-1}{3} - \binom{n-1}{2}$. Then one of the following cases must hold:
\begin{enumerate}
\item $t=2$ and $\lambda''= [s-1,s-1]$ and $s\leq 4$;
\item $2< t \leq 5$ and $\lambda''= [s^{t-1}, 1, 1]$ and $s\leq2$; or
\item $s=1$. 
\end{enumerate}
The first two cases imply that $n \leq 10$ and the final case implies that 
$\chi = [1^{(n+1)}]$ which has dimension 1.

Thus $\resg{\chi}{\sym{n-1}}$ has two representations with dimension at least 
$\binom{n-1}{3} - \binom{n-1}{2}$, so the dimension of $\chi$ is at least $(n-1)(n-2)(n-6)/3$.
which is strictly greater than $(n+1)n(n-4)/6$ for $n \geq 13$. This is a contradiction, so no such $\chi$ exists. 
\qed

Next we will show that there are only three irreducible representations in the decomposition of 
$\indg{ 1_{\sym{k} \wr \sym{\ell}}}{ \sym{k\ell} }$ that have dimension no more than $\binom{k\ell}{3}-\binom{k\ell}{2}$.
To do this we will consider the action of different Young subgroups on $\mathcal{U}_{k,\ell}$.
For any integer partition $\lambda \vdash n$ we will denote the Young subgroup by
\[
\sym{\lambda} = \sym{\lambda_1} \times \sym{\lambda_2} \times \dots \times \sym{\lambda_k}. 
\]

\begin{theorem}\label{thm:RepWeWant}
Assume $k\ell\geq 13$. Then the only partitions in the decomposition of $\indg{ 1_{\sym{k} \wr \sym{\ell}}}{ \sym{k\ell} }$ with dimension 
less than or equal to  $\binom{k\ell}{3}-\binom{k\ell}{2}$ are 
\[
\chi_{[k\ell]}, \quad \chi_{[k\ell-2,2]}, \quad \chi_{[k\ell-3,3]}.
\]
\end{theorem}

\begin{proof}
Lemma~\ref{lem:tenspecial}, lists the 10 irreducible representations of $\sym{k\ell}$ 
with dimension no more than $\binom{k\ell}{3}-\binom{k\ell}{2}$. We only need to show which of these representations are in the decomposition. The tool we use is Frobenius reciprocity along with the action of different Young subgroups on $\mathcal{U}_{k, \ell}$.

By Frobenius reciprocity
\begin{align*}
 \langle \, \indg{1_{ \sym{\lambda} }}{ \sym{k\ell} } , \quad \indg{1_{\sym{k} \wr \sym{\ell} }}{ \sym{k\ell} } \, \rangle_{\sym{k\ell}}=&\\
 \langle \, 1 ,  \quad \resg{  \indg{ 1_{\sym{k} \wr \sym{\ell} } } { \sym{k\ell} }  } {\sym{\lambda}}  \, \rangle_{\sym{\lambda}}.&
\end{align*}
The second inner product above gives the number of orbits of the action of $\sym{\lambda}$ on the 
cosets $\sym{k\ell}/( \sym{k} \wr \sym{\ell})$; or, equivalently, the number of orbits of $\sym{\lambda}$ on the partitions in $\mathcal{U}_{k, \ell}$.
Using this fact with different Young subgroups will allow us to determine that many of the representations 
with small degree do not occur in the decomposition of $\indg{ 1_{\sym{k} \wr \sym{\ell} }}{ \sym{k\ell} }$.

To start, it is clear that $\sym{k\ell}$ has one orbit on the $(k,\ell)$-partitions, 
so $\chi_{[k \ell]}$ has multiplicity 1 in the decomposition.
Next consider the group $\sym{[k\ell-1,1]}$, it is also straight-forward that this group only has one orbit on the partitions. 
The following decomposition is well-known
\[
\indg{ 1_{\sym{[k\ell-1,1]}} } {\sym{k\ell}} = \chi_{[k\ell]} + \chi_{[k\ell -1,1]},
\]
so we have that
\[
\langle \chi_{[k\ell]} + \chi_{[k\ell -1,1]} ,  \indg { 1_{\sym{k} \wr \sym{\ell} } } { \sym{k\ell} }  \rangle = 1
\]
Since we know that $\chi_{[k\ell]}$ occurs in this decomposition with multiplicity 1, this implies that 
$\chi_{[k \ell-1,1]}$ does not occur in the decomposition of $\indg{1_{\sym{k} \wr \sym{\ell} }}{ \sym{k\ell} } $.

Next we consider the group $\sym{[k\ell-2,2]}$. This group has two orbits on the partitions of $\mathcal{U}_{k, \ell}$. 
Using the well-known decomposition of $\indg{ 1_{\sym{[k\ell-2,2]}} } {\sym{k\ell}}$ we have that
\begin{align*}
& \langle \, \indg{ 1_{\sym{[k\ell-2,2]}} } {\sym{k\ell}}, \quad  \indg{ 1_{\sym{k} \wr \sym{\ell} }}{ \sym{k\ell} }  \, \rangle = \\
& \langle \, \chi_{[k\ell]} + \chi_{[k\ell -1,1]} + \chi_{[k\ell-2,2]}, \quad \indg{ 1_{\sym{k} \wr \sym{\ell} }}{ \sym{k\ell} } \, \rangle
 = 2.
\end{align*}
This implies that $\chi_{[k\ell-2,2]} $ occurs in the decomposition of $\indg{1_{\sym{k} \wr \sym{\ell} }}{ \sym{k\ell} }$ with multiplicity 1.

We continue this process with the group $\sym{[k\ell-2,1,1]}$. It has two orbits on the partitions of $\mathcal{U}_{k,\ell}$. Since
\[
\indg{  1_{\sym{[k\ell-2,1,1]}} }{\sym{k\ell}} 
    = \chi_{[k\ell]} + \chi_{[k\ell-1,1]} + \chi_{[k\ell-2,2]} + \chi_{[k\ell-2,1,1]},
\]
we can conclude that $\chi_{[k\ell-2,1,1]}$ does not occur in the decomposition. 

Next, we consider the group $\alt{k\ell} \cap \sym{[k\ell-2]}$. This group has two orbits on the partitions of $\mathcal{U}_{k, \ell}$.  Again the decomposition of $\indg{ 1_{ \alt{k\ell} \cap \sym{[k\ell-2]} } } {\sym{k\ell}} $ is well-known (a proof can be found in \cite[Proposition 1.4]{MR2581098}) and we have 
\begin{align*}
\indg{ 1_{\alt{k\ell} \cap \sym{[k\ell-2]}  } } {\sym{k\ell}} 
 &= \chi_{[k\ell]} + \chi_{[k\ell-1,1]} + \chi_{[k\ell-2,2]} + \chi_{[k\ell-2,1,1]} \\
 & \quad + \chi_{[1^{k\ell}]} + \chi_{[2, 1^{k\ell-2}]} + \chi_{[2,2, 1^{k\ell-4}]} + \chi_{[3, 1^{k\ell-3}]} 
\end{align*}
this implies that none of $\chi_{[1^{k\ell}]}$, $\chi_{[2, 1^{k\ell-2}]}$, $\chi_{[2,2, 1^{k\ell-4}]}$ and 
$\chi_{[3, 1^{k\ell-3}]}$ occur in the decomposition of $\indg{1_{\sym{k} \wr \sym{\ell} }}{ \sym{k\ell} }$.

Next we consider the group $\sym{[k\ell-3,3]}$. This group has three orbits on the partitions of $\mathcal{U}_{k,\ell}$ and from the 
decomposition of $\indg{ 1_{\sym{[k\ell-3,3]}} } {\sym{k\ell}}$ we have that
\begin{align*}
& \langle \indg{ 1_{ \sym{[k \ell-3,3]}} } {\sym{k \ell}}, \indg{ 1_{\sym{k}  \wr \sym{\ell} }}{ \sym{k \ell} }  \rangle =\\
& \langle \chi_{[k\ell]} + \chi_{[k\ell -1,1]} + \chi_{[k\ell-2,2]}+ \chi_{[k\ell-3,3]}, 
            \indg{   1_{\sym{k} \wr \sym{\ell}  } }  { \sym{k\ell} }  \rangle
 = 3.
\end{align*}
This implies that $\chi_{[k\ell-3,3]} $ occurs in the decomposition of $\indg{1_{\sym{k} \wr \sym{\ell}}}{ \sym{k\ell} }$ with multiplicity 1.

Next we consider the alternating group $\alt{k\ell} \cap \sym{[k\ell-2,1,1]}$. This group has 2 orbits on the partitions of 
$\mathcal{U}_{k, \ell}$ and 
\begin{align*}
\indg{ 1_{ \alt{k\ell} \cap \sym{[k\ell-2,1,1]} } }{\sym{k\ell}} 
  &= \chi_{[k\ell]} + \chi_{[k\ell-1,1]} + \chi_{[k\ell-2,2]} + \chi_{[k\ell-2,1,1]} \\
  & \quad \chi_{[1^{k\ell}]} + \chi_{[2, 1^{k\ell-2} ]} + \chi_{[2,2,1^{k\ell-4}]} + \chi_{[3, 1^{k\ell-3}]}.
  \end{align*}
Since $\chi_{[k\ell]}$, and $\chi_{[k\ell-2,2]}$ are in the decomposition, none of the irreducible representations
 $\chi_{[1^{k\ell}]}$, $\chi_{[2, 1^{k\ell-2}]}$, $\chi_{[2, 2, 1^{k\ell-4}]}$, or $\chi_{[3,1^{k\ell-3}]} $ occur in the decomposition of $\indg{1_{\sym{k} \wr \sym{\ell}}}{ \sym{k\ell} }$.

Finally, we consider the alternating group $\alt{k\ell} \cap \sym{[k\ell-3,3]}$. This group has three orbits on the partitions of 
$\mathcal{U}_{k, \ell}$ and 
\begin{align*}
\indg{ 1_{ \alt{k\ell} \cap \sym{[k\ell-3,3]} } }{\sym{k\ell}} 
  &= \chi_{[k\ell]} + \chi_{[k\ell-1,1]} +\chi_{[k\ell-2,2]} +\chi_{[k\ell-3,3]} \\
 & \quad + \chi_{[1^{k\ell}]} + \chi_{[2, 1^{k\ell-2}]} +\chi_{[2^2, 1^{k\ell-4}]} +\chi_{[2^3, 1^{k\ell-6}]}.
\end{align*}
Which shows $\chi_{[2,2,2,1^{k\ell-6}]} $ is not in the decomposition of $\indg{1_{\sym{k} \wr \sym{\ell}}}{ \sym{k\ell} }$.
\end{proof}

\section{Eigenvalues of $X_{k,\ell}$ with $k \geq 3$}\label{sec:evals}

In this section we will find three of the eigenvalues of $X_{k,\ell}$. For ease of notation, we will denote the irreducible representation of $\chi_\lambda$ by the $\lambda$-module. Also, the number of vertices in $X_{k,\ell}$,  which is equal to $u_{k,\ell}$, will be denoted simply by $v$ and the degree of the graph $X_{k, \ell}$ will be simply written as $d$, rather than $d_{k,\ell}$. 

Any subgroup $H \leq \sym{k\ell}$ acts on the vertices of $X_{k, \ell}$ and the orbits of this action for an equitable partition. From any equitable partition, we can form a quotient graph and the eigenvalues of this quotient graph will be eigenvalues of the $X_{k, \ell}$ (details can be found in~\cite[Section 2.2]{MR3646689}). 
The trivial case is $H = \sym{k\ell}$, since this group is transitive, the equitable partition has all the 
vertices of $X_{k, \ell}$ in a single part. The quotient graph for this simply the $1 \times 1$ matrix with the single entry 
$d$.  The eigenvalue of this matrix is simply $d$, and the eigenvector is the all ones vector and the eigenspace is 
isomorphic to the trivial representation of $\sym{k\ell}$. So $d$ belongs to the $[k\ell]$-module.

Since the subgroup $\sym{[k\ell-1,1]}$ has only one orbit on the vertices of $X_{k,\ell}$, the next subgroup we consider 
is the Young subgroup $\sym{[k\ell -2,2]}$, considered as the stabilizer of the set $\{1,2\}$. This subgroup is not 
transitive on the partitions, it has exactly 2 orbits: 
$S_1$ the set of all partitions that have 1 and 2 in the same block, and $S_2$ the 
set of all partitions in which 1 and 2 are in different blocks. The first orbit, $S_1$ is a coclique in $X_{k,\ell}$ so the quotient 
matrix for this partition has the form
\[
\begin{pmatrix}
0 & d \\
-\tau & d+\tau 
\end{pmatrix} .
\]
The eigenvalues of this quotient graph are $d$ and $\tau$. We can calculate the value of $\tau$ by counting edges between $S_1$ and $S_2$. Since $S_1$ is a coclique, each vertex in $S_1$ is adjacent to $d$ vertices in $S_2$, and each vertex in $S_2$ is adjacent to $-\tau$ vertices in $S_1$. Using the sizes of $S_1$ and $S_2$, we have that the number of edges between $S_1$ and $S_2$ is equal to
\[
|S_1| d = \binom{k\ell-2}{k-2} u_{k,\ell-1} d
\]
and also to
\[
|S_2|(-\tau) =  \binom{k\ell -2}{k-1}\binom{k\ell-k-1}{k-1} u_{k,\ell-2} (-\tau).
\]
Thus 
\begin{equation}\label{eq:tau}
\tau = -\frac{(k-1)d}{k(\ell-1)}
\end{equation}
is a second eigenvalue for $X_{k, \ell}$. Since this eigenvalue arises from the action of $\sym{[k\ell-2,2]}$, it belongs
to a module that is common between the two representations
\[
 \indg{ 1_{\sym{k}  \wr \sym{\ell} }}{ \sym{k \ell} } 
 \quad 
 \quad
  \indg{ 1_{\sym{ [k \ell - 2, 2] } } } { \sym{k \ell} } 
\]
Thus it belongs to the module $[k\ell-2,2]$, as this is the only common module, and must have dimension at least $\binom{k\ell}{2} - \binom{k\ell}{1}$. (A second irreducible module could also have $\tau$ as the eigenvalue belonging to it, so the dimension could be higher.)

\begin{lemma}\label{lemma:multTau}
For integers $k$ and $\ell$, with $k,\ell  \geq 2$, 
$\tau  = -\frac{(k-1)d}{k(\ell-1)}$ is an eigenvalue of $X_{k,\ell}$ with multiplicity at least $\binom{k\ell}{2} - \binom{k\ell}{1}$.\qed
\end{lemma}

Next we will consider the Young subgroup $\sym{ [k\ell-3,3] }$, thought of as the group that stabilizes the set $\{1,2,3\}$. 
The action of this subgroup on $\mathcal{U}_{k,\ell}$ has 3 orbits: $T_1$, the set of all partitions with $1, 2,3$ in the same block; $T_2$ the set of all partitions in which $1, 2, 3$ are in exactly two different blocks; and $T_3$ the set of all partitions in which $1, 2, 3$  are in three different blocks. Any vertex in $T_1$ is adjacent only to vertices in $T_3$. Similarly, a vertex in $T_2$  can be adjacent to vertices in $T_2$ and $T_3$. The quotient graph for this equitable partition is 
\[
M = \begin{pmatrix}
0 & 0 & d \\
0 & a & d-a  \\
b & c & d-b-c
\end{pmatrix} .
\]
where $a,b,c$ are all non-negative. 

The eigenvalues for this quotient graph will be the eigenvalues that belong to modules that are both the decomposition of 
$\indg{1_{\sym{[k\ell-3,3]}}}{\sym{k\ell}}$ and the decomposition of $\indg{1_{\sym{k} \wr \sym{\ell} }}{\sym{k\ell}}$. 
Thus the eigenvalues will belong to the $[k\ell]$, $[k\ell-2,2]$ and $[k\ell-3,3]$ modules. We have already seen that 
the eigenvalue for $[k\ell]$ is $d$, and the eigenvalue for $[k\ell-2,2]$ is $\tau$. We will denote the eigenvalue belonging to $[k\ell-3,3]$ by $\theta$.

Since the trace of the matrix is the sum of the eigenvalues we have that
\begin{align}\label{eq:trace}
d+a-b-c = d+\tau+\theta.
\end{align}

The number of edges between $T_1$ and $T_3$ is equal to 
\[
 d |T_1|  =  d \, \binom{k\ell-3}{k-3} u_{k,\ell-1} , 
\] 
 and also to
\[
b |T_3| =  b \,  \binom{k\ell-3}{k-1}\binom{k\ell-k-2}{k-1}\binom{k\ell-2k-1}{k-1} u_{k,\ell-3}.
 \]
Setting these equations equal to each other, then expanding the binomial coefficients and rearranging yields
\[
\frac{ (k-1)(k-2) }{k^2(\ell-1)(\ell-2)}  d =   b. 
\]
Replacing $d = -\frac{k(\ell-1)}{k-1} \tau$ shows that
\begin{align}\label{eq:bvalue}
b = -\frac{ (k-1)(k-2) }{k^2(\ell-1)(\ell-2)}   \frac{k(\ell-1)}{(k-1)} \tau = -\frac{k-2 }{k (\ell-2)}  \tau.
\end{align}
Put this into Equation~\ref{eq:trace} to get the following formula for $\theta$
\begin{equation}\label{eq:theta2}
\theta = a+\frac{k-2 }{k (\ell-2)}  \tau-c-\tau 
= a -c + \frac{ (k-2) - k(\ell-2) }{k (\ell-2)}  \tau.
\end{equation}

Similarly, counting the number of edges between $T_2$ and $T_3$ yields
\[
3 \binom{k\ell-3}{k-2}\binom{k\ell-k-1}{k-1} u_{k,\ell-2} (d-a) 
= \binom{k\ell-3}{k-1} \binom{k\ell - k -2}{k-1} \binom{k\ell-2k-1}{k-1} u_{k,\ell-3} (c) \\ 
\]
Again, expanding the binomial coefficients and rearranging shows that
\[
a =   d- \frac{(\ell-2)k}{3(k-1)} c.
\]

The characteristic polynomial of $M$ is
\[
  x^3 + (-a + b + c - d)x^2 + (-ab + ad - bd - cd)x + abd
\]
Substituting in the values we have computed for $b$ and $c$, and using the fact that $\tau$ is a root of the 
characteristic polynomial we get
\begin{equation}\label{eq:avalue}
  a =  \frac{2(k-1)}{k(\ell-1)}d.
\end{equation}
From this we can compute that 
\begin{equation}\label{eq:cvalue}
c = \frac{3(k\ell-3k+2)(k-1)}{k^{2}(\ell-1)(\ell-2)}d.
\end{equation}

\begin{lemma}\label{lemma:multTheta}
For integers $k$ and $\ell$, with $k,\ell \geq 3$, 
\[
\theta = \frac{2(k-1)(k-2)d}{k^{2}(\ell-1)(\ell-2)}
\]
is an eigenvalue of $X_{k,\ell}$ with multiplicity at least $\binom{k\ell}{3} - \binom{k\ell}{2}$.
\end{lemma}
\begin{proof}
By Equations~(\ref{eq:theta2}),~(\ref{eq:avalue})  and~(\ref{eq:cvalue}), we can calculate that
\begin{equation}\label{eq:thetafinal}
\theta = \frac{2(k-1)(k-2)d}{k^{2}(\ell-1)(\ell-2)}.
\end{equation}

From the comments above, $\theta = \frac{2(k-1)(k-2)d}{k^{2}(\ell-1)(\ell-2)}$ is the eigenvalue belonging to the unique $[k\ell-3,3]$-module in $ \indg{ 1_{\sym{k}  \wr \sym{\ell} }}{ \sym{k \ell} } $.
Since the dimension of the irreducible representation of $[k\ell-3,3]$ is $\binom{k\ell}{3} - \binom{k\ell}{2}$, the multiplicity 
of $\theta$ is at least $\binom{k\ell}{3} - \binom{k\ell}{2}$.
\end{proof}

\section{Bound on degree of $X_{k,\ell}$}\label{sec:degree}

In this section we will find a lower bound the degree of $X_{k,\ell}$ for all sufficiently large $\ell$. If $P$ and $Q$ are two partitions that are adjacent in $X_{k,\ell}$, then the meet table of $P$ and $Q$ is an $\ell \times \ell$ matrix with entries either 0 or 1, and further, the entries in each row and column in the meet table sum to $k$. We define $\mathcal{M}_{k,\ell}$ to be the set of all such meet tables, so all $\ell \times \ell$ matrices with entries either 0 or 1, and row and columns sums equal to $k$. To find the degree of $X_{k,\ell}$, we first, we state a result on the number of such meet tables. Next, for a fixed partition $P$ and a meet table $M \in \mathcal{M}_{k,\ell}$, we count the number of partitions $Q$ for which the meet table of $P$ and $Q$ is $M$. 

Bender~\cite{MR389621} determined the asymptotic cardinality of $\mathcal{M}_{k,\ell}$. (In fact, Bender found a much more general result, but we only state the result that we need here.)

\begin{theorem}[\cite{MR389621}]\label{thm_bender}
For positive integers $k,\ell$ 
\[
    \lim_{\ell \rightarrow \infty} \cardinality{\mathcal{M}_{k,\ell}} = \frac{(k\ell)!}{(k!)^{2\ell}}e^{-\frac{(k-1)^2}{2}}. \qed
\]
  \end{theorem}

To get a lower bound on $d_{k, \ell}$, we fix a partition $P$ in $\mathcal{U}_{k, \ell}$, then for each $M \in \mathcal{M}_{k,\ell}$, we will count the number of $Q$ so that the meet table of $P$ and $Q$ is $M$, then we use Theorem~\ref{thm_bender} to bound the size of $\mathcal{M}_{k,\ell}$.

\begin{lemma}\label{degree_lemma}
For positive integers $k, \ell$ with $k\leq \ell$, 
  \[
    d_{k,\ell} = \frac{ k!^\ell}{\ell!} \cardinality{\mathcal{M}_{k,\ell} }.
  \]
\end{lemma}

\begin{proof}
Fix a partition $P \in \mathcal{U}_{k, \ell}$. Define a bipartite multigraph with the vertices in one part the set 
$\mathcal{M}_{k,\ell}$, and the vertices in the other part the neighbourhood of $P$ in $X_{k,\ell}$. 
Two vertices $M$ and $Q$ are adjacent if  the meet table of $P$ and $Q$ is $M$.
By counting the number of edges in this graph in two ways, we will determine the size of the neighbourhood of $P$ in terms of $\cardinality{\mathcal{M}_{k,\ell} }$.

For any $M \in \mathcal{M}_{k,\ell}$, with $M = [m_{i,j}]$ assume that row $i$ correspond to the block $P_i \in P$.
Construct a partition $Q = \{Q_1,Q_2,\dots,Q_\ell\}$ so that the block $Q_j$ corresponds to column $j$ of $M$ and 
$|P_i \cap Q_j | = m_{i,j}$. 
Since the entries of a row in $M$ are either 0 or 1, and sum to $k$, there are $k!$ ways to select how the elements from $P_i$ will be distributed to the blocks of $Q$.  So for each meet table $M$, there are $k!^{\ell}$ partitions $Q$ that can be constructed this way. It is possible that some of these partitions are equal, once the blocks are reordered, so this is a multigraph.   
  
For every $Q$ in the neighbourhood of $P$, there are $\ell!$ ways to order the blocks of $Q$, once the blocks are ordered the meet table for $P$ and $Q$ is uniquely defined. In the bipartite graph, $Q$ is adjacent to each of these tables in the graph (again, these tables may not be distinct, so the graph is a multigraph). The degree of every vertex $Q$ is $\ell!$. 

Thus we have that the number of edges in the multigraph is
    \[
    \ell! d_{k,\ell} =  \sum_{M \in \mathcal{M}_{k,\ell} } k!^\ell,
  \]
  and the result follows.
\end{proof}

Using Theorem~\ref{thm_bender} we have the asymptotic size of $d_{k,\ell}$. 

\begin{corollary}\label{cor:uoverd}
  For a fixed integer $k$ with $k\geq 2$,
  \[
    \lim_{\ell \rightarrow \infty}  \frac{u_{k,\ell}}{d_{k,\ell}} = e^{\frac{(k-1)^2}{2}}.
  \]
 \end{corollary}
 
 \begin{proof}
 This follows from the value of $u_{k,\ell}$ given in Equation~(\ref{eq:ukl}) and from the fact that
\[
\lim_{\ell \rightarrow \infty}  d_{k,\ell} = \frac{ (k\ell)! }{ (k!)^{\ell}\ell! } e^{\frac{-(k-1)^2}{2}}.
\]
\end{proof}

Thus for every $\epsilon > 0$, there exists an $\ell'$ such that for all $\ell \geq \ell'$,
\[
  \frac{u}{d} \leq e^{\frac{(k-1)^2}{2}} + \epsilon.
  \]


\section{A bound on the multiplicity of eigenvalues with large absolute value}\label{sec:multiplicity}

In Section~\ref{sec:evals} we found three eigenvalues, $d$, $\tau$, and $\theta$ of $X_{k,\ell}$. The ratio $\frac{d}{\tau} = \frac{k(1-\ell)}{k-1}$, so 
\[
\frac{  | V(X_{k,\ell}) | }{ 1- \frac{d}{\tau}} = 
\frac{  | V(X_{k,\ell}) | }{ 1 - \frac{k(1-\ell)}{k-1} } 
= u_{k, \ell-1}.
\]
This is exactly the size of a set of canonically $2$-intersecting $(k,\ell)$-partitions.
If we can show that $\tau$ is the least eigenvalue of $X_{k,\ell}$, then the ratio bound implies that these are cocliques of maximum size. In this section we show if $X_{k,\ell}$ has an eigenvalue $\lambda$ with $\lambda^{2}>\tau^{2}$, then there is a bound on the multiplicity of $\lambda$.

Let 
\[
\{d^{(1)}, \tau^{(m_{\tau})}, \theta^{(m_{\theta})}, \lambda_{2}^{(m_{2})},\hdots, \lambda_{j}^{(m_{j})}  \}
\]
be the spectrum of the matrix $X_{k,\ell}$, where the values $m_{i}$ represent the multiplicities of the eigenvalues. 
By squaring $A$ and taking the trace, we have
\begin{equation*}
vd = d^{2}+m_{\tau}\tau^{2}+m_{\theta}\theta^{2}+\sum_{i=2}^{j} m_{i}\lambda_{i}^{2}.
\end{equation*}
Hence for every $2\leq i \leq j$ we have
\begin{equation*}
vd-d^{2}-m_{\tau}\tau^{2}-m_{\theta}\theta^{2}\geq m_{i}\lambda_{i}^{2}.
\end{equation*}

Assume $\lambda_i$ is an eigenvalue of $X_{k,\ell}$ with $\lambda_i^{2}>\tau^{2}$, and also that $\lambda_i$ is not the eigenvalue belonging to the  $[k\ell]$, $[k\ell-2,2]$ or $[k\ell-3,3]$ modules, 
then 
\[
\frac{vd-d^{2}-m_{\tau}\tau^{2}-m_{\theta}\theta^{2}}{\tau^{2}}\geq m_{i} 
\]
Expanding $\theta$ using Equation~(\ref{eq:thetafinal}) in the above equation produces the following equation
\begin{equation*}
\left( \frac{v}{d}-1 \right) \frac{k^{2}(\ell-1)^2}{(k-1)^2}-m_{\theta}\frac{4(k-2)^{2}}{k^{2}(\ell-2)^{2}}-m_{\tau}\geq m_{i}.
\end{equation*}
Further, by Lemmas~\ref{lemma:multTau} and \ref{lemma:multTheta}, it is known that $m_{\tau}\geq \binom{k\ell}{2}-\binom{k\ell}{1}$ and $m_{\theta}\geq \binom{k\ell}{3}-\binom{k\ell}{2}$, so this bound becomes 
\begin{equation*}
\left( \frac{v}{d}-1 \right)
\frac{k^{2}(\ell-1)^2}{(k-1)^2}-\frac{(k\ell)(k\ell-1)(k\ell-5)}{6}\frac{4(k-2)^{2}}{k^{2}(\ell-2)^{2}}-\frac{(k\ell)(k\ell-3)}{2}\geq m_{i}.
\end{equation*}

Our next step is to show that this upper bound on $m_{i}$ is smaller than $\binom{k\ell}{3}-\binom{k\ell}{2}$. This will be a contradiction since have assumed that $\lambda$ does not belong to any of the $[k\ell]$, $[k\ell-2,2]$, and $[k\ell-3,3]$ modules. In other words, we need to prove that
\begin{equation}\label{eqn_needed}
\frac{v}{d}-1 <\frac{\ell(k-1)^{2}}{6k^{3}(\ell-1)^{2}(\ell-2)^{2}} \left( k^{2}(\ell-2)^{2}(k\ell-4)(k\ell+1)+4(k-2)^{2}(k\ell-1)(k\ell-5) \right).
\end{equation}
This will follow from Corollary~\ref{cor:uoverd}.
 
\begin{theorem}\label{thm:main}
      Fix an integer $k \geq 3$. For $\ell$ sufficiently large, the largest set of partially 2-intersecting uniform $(k,\ell)$-partitions has size
      \[
        \binom{k\ell-2}{k-2} u_{k,\ell-1}
      \]
\end{theorem}
\begin{proof}
For any distinct $i, j \in \{1,\dots, k\ell\}$, the set $S_{i, j}$ of all $(k,\ell)$-partitions with $i$ and $j$ are in the same block form a set of partially 2-intersecting $(k,\ell)$-partitions of the size given in the theorem.

Corollary~\ref{cor:uoverd} shows that $\frac{v}{d}$ approaches a fixed constant, namely $e^{\frac{(k-1)^2}{2}}$, as $\ell$ goes to infinity. Since the right hand side of Equation~(\ref{eqn_needed}) grows linearly in $\ell$, we have that 
Equation~(\ref{eqn_needed}) holds for $\ell$ sufficiently large. This implies if there is an eigenvalue $\lambda$ 
of $X_{k,\ell}$ with $\lambda \leq \tau$, 
 then the multiplicity of $\lambda$ is less than or equal to $\binom{k\ell}{3}-\binom{k\ell}{2}$.
 
By Theorem~\ref{thm:RepWeWant}, eigenspaces with dimension less than or equal to 
$\binom{k\ell}{3}-\binom{k\ell}{2}$ can only include the $[k\ell]$, $[k\ell-2,2]$ or the $[k\ell-3,3]$-modules. 
The degree, $d$, is the eigenvalue belonging to the $[k\ell]$-module, and Lemma~\ref{lemma:multTau} and 
Lemma~\ref{lemma:multTheta} give the eigenvalues belonging to the $[k\ell-2,2]$ or the $[k\ell-3,3]$-modules. 
So we can conclude that $\tau = -\frac{(k-1)d}{k(\ell-1)}$ is the least eigenvalue of $X_{k,\ell}$ and that $\tau$ 
belongs only to the $[k\ell-2,2]$-module. 

By the ratio bound, Theorem~\ref{Thm:RatioBound}, the maximum size of coclique in $X_{k,\ell}$ is 
\[
\frac{|V( X_{k,\ell} )|}{1- \frac{d}{\tau}} 
= \frac{ v }{1 - \frac{d}{ -\frac{(k-1)d}{k(\ell-1)} }}
= \frac{ v }{1 + \frac{k(\ell-1)}{ k-1} }
=  \frac{v  (k-1)  }{ k\ell - 1}
= \binom{k\ell-2}{k-2} u_{k,\ell-1}.
\]
\end{proof}

The previous result shows that the sets $S_{i,j}$ are the largest intersecting sets. We further conjecture
 that these sets are the only maximum intersecting sets.

\begin{conj}
For $k\geq 3$ and $\ell$ sufficiently large, the only sets of partially 2-intersecting $(k,\ell)$-partitions with size
$\binom{k\ell-2}{k-2} u_{k,\ell-1}$ are the sets $S_{i, j}$.
\end{conj}

We can make a step towards this conjecture with the following weaker characterization of the maximum intersecting sets.
Denote the characteristic vectors of the sets $S_{i,j}$ by $v_{i,j}$.

\begin{cor}
For a fixed integer $k \geq 3$ and $\ell$ sufficiently large, let $S$ be any maximum partially 2-intersecting set of $(k,\ell)$-partitions. Then the characteristic vector of $S$ is a linear combination of the vectors $v_{i,j}$.
\end{cor}
\begin{proof}
For $k\geq 3$ and $\ell$ sufficiently large, $S_{i,j}$ is a maximum coclique in $X_{k,\ell }$ and equality holds in the ratio bound. Let $v_{i,j}$ be the characteristic vector of $S_{i,j}$. Since we have equality in the ratio bound, this implies that 
\[
v_{i,j}  - \frac{k-1}{k\ell-1} \one
\]
is a $\tau$-eigenvector. Since no other modules have have eigenvalue $\tau$, these vectors are in the $[k\ell-2,2]$-module. Further, the set of vectors 
\[
\left \{ v_{i,j}  - \frac{k-1}{k\ell-1} \one \, | \, i, j \in \{1,\dots, k\ell\}  \right \}
\]
is invariant under the action of $\sym{k\ell}$, so they form a module. Since the $[k\ell-2,2]$-module is irreducible, these vectors span the entire $[k\ell-2,2]$-module; this also implies that the vectors 
$\{ v_{i,j}  \, | \, i, j \in \{1,\dots, k\ell\} \}$  span the $[k\ell]$ and $[k\ell-2,2]$-modules.

Let $S$ be a partially 2-intersecting set of $(k,\ell)$-partition of maximum size, and let $v_S$ denote the characteristic vector of $S$. Then $v_S - \frac{k-1}{k\ell-1} \one$ is in the $[k\ell-2,2]$-module. Thus $v_S$ is in the span of the $[k\ell]$ and $[k\ell-2,2]$-module, so $v_S$ is a linear combination of the $v_{i,j}$. 
\end{proof}
 
\section{Exact result for $k=3$}
\label{sec:exact3}

Corollary 7.5.6 in~\cite{MR2708226} proves Theorem~\ref{thm:main} holds for $k=3$ and $\ell$ odd. In this section
 we will prove that the theorem actually holds for all $\ell \geq 3$ with $k=3$. For $k=3$, we observed experimentally 
 that the ratio $u_{3,\ell}/d_{3,\ell}$ converges to $e^{\frac{(k-1)^2}{2}} = e^2$ surprisingly quickly. If the sequence of $u_{3,\ell}/d_{3,\ell}$ was non-increasing this would be sufficient, but we have no proof of this. Rather, in this section we show an upper bound on the ratio $u_{3,\ell}/d_{3,\ell}$ for all $\ell$, or, equivalently, a lower bound on $d_{3,\ell}$. This bound holds for $\ell >10$, and we simply directly check the theorem for the specific graphs with smaller values of $\ell$.

\begin{lemma}\label{k_3_lemma}
  For $\ell > 10$, the degree, $d_{3,\ell}$ is greater than  $u_{3,\ell}/24$.
\end{lemma}
\begin{proof}
  We will use a truncated inclusion-exclusion argument to bound the degree. Since $X_{k,\ell}$ is vertex transitive, 
  we again obtain the bound on the degree by counting the neighbours of an arbitrary partition $P \in \mathcal{U}_{3, \ell}$. 
  Let $\mathcal{J}$ be the set of the $3\ell$ pairs $\{x,y\}$ of elements in $\{1,2,\dots,3\ell\}$ that are in the same block of $P$.
  For a pair $\{x, y\} \in \mathcal{J}$, we let $A_{\{x,y\}}$ be the set of all partitions which contain $x$ and $y$ in the same block. 
  For a subset $J \subseteq \mathcal{J}$, define 
  \[
  N(J) = \cardinality{  \cap_{j \in J} A_{j} }
  \]
  and for $0 \leq j \leq 3\ell$ let $N_j = \sum_{J, \cardinality{J} = j} N(J)$. By inclusion-exclusion,
  \begin{equation}\label{eq:altsum}
    d_{3,\ell} = \sum_{j=0}^{3\ell} -1^jN_j.
  \end{equation}
Next we calculate $N_j$. First, we note that $N_0 = N(\emptyset) =  u_{3,\ell}$.

 For any set $J$ and each block of $P$, there are either 0, 1, 2 or 3 pairs in $J$ which are contained in the block. Let $n_i$ be the number of blocks of $P$ that have $i$ of their pairs in $J$.  We call $(n_0,n_1,n_2,n_3)$ the {\em pair distribution of $J$} and note that $n_0+n_1+n_2+n_3 = \ell$. If a block of $P$ has $i$ of its pairs in a $J$, we say the block 
 has \textsl{type $n_i$}.
  
 To find $N_j$, we first fix a $J$ with a given pair distribution and count then number of $(3,\ell)$-partitions
  $Q$ in which every pair from $J$ is contained in some block of $Q$. Next, for a given pair distribution we count the number of sets $J \in \mathcal{J}$ that have the fixed pair distribution. Finally we count all the possible pair distributions.
  
First, fix a $J$ with pair distribution $(n_0,n_1,n_2,n_3)$ and count the number of partitions $Q \in  \cap_{j \in J} A_{j} $.
The blocks of type $n_3$ determine exactly which three elements are in a block of $Q$, as do the blocks of type $n_2$.
Each of the blocks of type $n_1$ contain one pair and determine two of the three points in their respective blocks.  One more point must be chosen to complete each of these and this choice is ordered since each pair of type $n_1$ from $J$ uniquely labels its corresponding block. Each of the blocks of type $n_0$ don't determine any points in $Q$. Thus 
the number of partitions $\mathcal{Q}$ which contain the pairs from $J$ is given by the multinomial coefficient
  \[
    \frac{1}{n_0!} \binom{3\ell-3(n_3+n_2)-2n_1}{1^{n_1}, 3^{n_0}}. 
  \]

We now count the number of possible $J$ which have pair distribution $(n_0,n_1,n_2,n_3)$.  The number of ways to select
the type of each block in $P$ is equal to the multinomial coefficient
  \[
    \binom{\ell}{n_0,n_1,n_2,n_3},
  \]
since we are choosing the four sets of blocks from $\mathcal{P}$ that have either 0, 1, 2 or 3 pairs in $J$. 
Each of the blocks of $P$ with $n_3$ pairs has all of its three pairs in $J$.  For each of the blocks with $n_2$ pairs, there are three ways to choose which two of the three possible pairs are in $J$. For each of the blocks with $n_1$ pairs in $J$, there are three ways to chose which one of the pairs is in $J$.  Finally, each of the $n_0$ blocks do not contribute any pairs to $J$.  Thus there are
  \[
    3^{n_1+n_2}  
  \]
  different sets $J$ once the blocks of $\mathcal{P}$ are assigned to the four sets.

Each pair distribution $(n_0,n_1,n_2,n_3)$ is a composition, that is, an ordered partition of $\ell$ into exactly four non-negative parts.  The pair distribution $(n_0,n_1,n_2,n_3)$ corresponds to a set $J$ of size $n_1+2n_2+3n_3$.  
Define $\mathcal{C}(\ell,j)$ to be the set of compositions of $\ell$ into four parts with $n_1 + 2n_2 + 3n_3 = j$. 
 
 Then from our previous counting we have that  
  \[
    N_j =  \sum_{(n_0,n_1,n_2,n_3) \in \mathcal{C}(\ell,j)}  3^{n_1+n_2}\binom{\ell}{n_0,n_1,n_2,n_3} \frac{1}{n_0!}\binom{3\ell-3(n_3+n_2)-2n_1}{1^{n_1}, 3^{n_0}}.
  \]

When we put this value in Equation~\ref{eq:altsum} and truncate this sum after an odd $j$ we will get a lower bound on $d_{3,\ell}$.  Taking $j$ up to 5 we sum over the following list of pair distributions:
          \begin{align*}
            \mathcal{C}(\ell,0) &= \{(\ell,0,0,0)\} \\
            \mathcal{C}(\ell,1) &= \{(\ell-1,1,0,0)\} \\
            \mathcal{C}(\ell,2) &= \{(\ell-1,0,1,0),(\ell-2,2,0,0)\} \\
            \mathcal{C}(\ell,3) &= \{(\ell-1,0,0,1),(\ell-2,1,1,0),(\ell-3,3,0,0)\} \\
            \mathcal{C}(\ell,4) &= \{(\ell-2,1,0,1),(\ell-2,0,2,0),(\ell-3,2,1,0),(\ell-4,4,0,0)\} \\
            \mathcal{C}(\ell,5) &= \{(\ell-2,0,1,1),(\ell-3,2,0,1),(\ell-3,1,2,0),(\ell-4,3,1,0),(\ell-5,5,0,0)\}             
            \end{align*}
 Expanding this becomes           
          \begin{align*}
            d_{3,\ell} &\geq \sum_{j=0}^{5} -1^{j} \sum_{(n_0,n_1,n_2,n_3) \in \mathcal{C}(\ell,j)}  3^{n_1+n_2}\binom{\ell}{n_0,n_1,n_2,n_3} \frac{\binom{3\ell-3(n_3+n_2)-2n_1}{1^{n_1}, 3^{n_0}}}{n_0!}\\                      
&=       \frac{\binom{\ell}{\ell}\binom{3\ell}{3^{\ell}}}{(\ell)!} + \frac{-3\binom{\ell}{\ell-1,1}\binom{3\ell-2}{1,3^{\ell-1}}}{(\ell-1)!} + \frac{3\binom{\ell}{\ell-1,1}\binom{3\ell-3}{3^{\ell-1}}}{(\ell-1)!} + \frac{3^2\binom{\ell}{\ell-2,2}\binom{3\ell-4}{1^{2},3^{\ell-2}}}{(\ell-2)!} \\
&\qquad +\frac{-\binom{\ell}{\ell-1,1}\binom{3\ell-3}{3^{\ell-1}}}{(\ell-1)!} + \frac{-3^2\binom{\ell}{\ell-2,1^{2}}\binom{3\ell-5}{1,3^{\ell-2}}}{(\ell-2)!} + \frac{-3^3\binom{\ell}{\ell-3,3}\binom{3\ell-6}{1^{3},3^{\ell-3}}}{(\ell-3)!} \\
&\qquad +\frac{3\binom{\ell}{\ell-2,1^{2}}\binom{3\ell-5}{1,3^{\ell-2}}}{(\ell-2)!} + \frac{3^2\binom{\ell}{\ell-2,2}\binom{3\ell-6}{3^{\ell-2}}}{(\ell-2)!} + \frac{3^3\binom{\ell}{\ell-3,2,1}\binom{3\ell-7}{1^{2},3^{\ell-3}}}{(\ell-3)!} \\
&\qquad +\frac{3^4\binom{\ell}{\ell-4,4}\binom{3\ell-8}{1^{4},3^{\ell-4}}}{(\ell-4)!} + \frac{-3\binom{\ell}{\ell-2,1^{2}}\binom{3\ell-6}{3^{\ell-2}}}{(\ell-2)!} + \frac{-3^2\binom{\ell}{\ell-3,2,1}\binom{3\ell-7}{1^{2},3^{\ell-3}}}{(\ell-3)!} \\
                       &\qquad +\frac{-3^3\binom{\ell}{\ell-3,1,2}\binom{3\ell-8}{1,3^{\ell-3}}}{(\ell-3)!} + \frac{-3^4\binom{\ell}{\ell-4,3,1}\binom{3\ell-9}{1^{3},3^{\ell-4}}}{(\ell-4)!} + \frac{-3^5\binom{\ell}{\ell-5,5}\binom{3\ell-10}{1^{5},3^{\ell-5}}}{(\ell-5)!} \\
            &= \frac{243\ell^6 - 2997\ell^5 + 13905\ell^4 - 32355\ell^3 + 42732\ell^2 - 32728\ell + 11200)(3\ell - 10)!}{80(6^{\ell-4})(\ell-10)!(\ell^6 - 39\ell^5 + 625\ell^4 - 5265\ell^3 + 24574\ell^2 - 60216\ell + 60480)}.
          \end{align*}
         
          Thus
          \[
            \frac{u_{3,\ell}}{d_{3,\ell}} \leq \frac{5(729\ell^6 - 6561\ell^5 + 23085\ell^4 - 40095\ell^3 + 35586\ell^2 - 14904\ell + 2240)}{243\ell^6 - 2997\ell^5 + 13905\ell^4 - 32355\ell^3 + 4273\ell^2 - 32728\ell + 11200}
            \]
            For $\ell > 10$ this gives that  $u_{3,\ell}/d_{3,\ell}  < 24$.
          \end{proof}

    \begin{theorem}
      For $k=3$ and all $\ell\geq 3$ the largest set of partially 2-intersecting uniform partitions has size
      \[
        (3\ell-2) u_{3,\ell-1}.
      \]
    \end{theorem}          
    \begin{proof}
    For $\ell = 3$ all the eigenvalues of $X_{3,3}$ have long been known to be $\{36,8,2,-4,-12\}$~\cite{MR773559} . The ratio bound holds with equality, and the only irreducible representation that belongs to the least eigenvalue is $\chi_{[7,2]}$.
    
    For $\ell=4$, all the eigenvalues of $X_{3,4}$ are $\{ 1296, 96, 72, 48, 32, 0, -24, -48, -288 \}$. These can be calculated by making a quotient graph of $X_{3,4}$ from the action of $\sym{3} \wr \sym{4}$ on the partitions. This equitable partition has a cell of size 1, so the eigenvalues of the quotient graph are exactly the eigenvalues of $X_{3,4}$. Further, the multiplicities of the eigenvalues can be calculated using the formulas in~\cite[Section 5.3]{MR1220704} and the $[10,2]$ module is the only module to which the eigenvalue $-288$ belongs.
        
    For $\ell \in \{5,\dots,12\}$ the only irreducible representations with dimension less then $\binom{3\ell}{3} - \binom{3\ell}{2}$
    in the decomposition of $\indg{ 1_{\sym{3} \wr \sym{\ell}}}{ \sym{3\ell} }$ are the three listed in Theorem~\ref{thm:RepWeWant}---this can be checked using GAP~\cite{GAP4}. Thus Theorem~\ref{thm:RepWeWant} holds for all $5 \leq \ell \leq 12$ when $k=3$.
           
      For all $\ell > 10$, Lemma~\ref{k_3_lemma} shows that $ u_{k,\ell}/d_{k,\ell} -1 < 23$.  In this same range, the right had side of Equation~(\ref{eqn_needed}) is at least 26.  Thus the inequality from Equation~(\ref{eqn_needed}) holds for all 
      $\ell > 10$.  
      
For $5 \leq \ell \leq 10$ the degrees $d_{3,\ell}$ can be directly computed  
\begin{center}
\begin{tabular}{lll}
$d_{3,5}  = 132192$, & $d_{3,7} = 3829057920$, & $d_{3,9} = 333973115062272$, \\
$d_{3,6} = 19258560$,  & $d_{3,8} = 1001695548672$,  & $d_{3,10} = 138348645213579264$,
 \end{tabular}
\end{center}
and the inequality from Equation~(\ref{eqn_needed}) directly checked.
\end{proof}

\section{Further work}

In this paper we only consider partially 2-intersecting partitions, but the conjecture in~\cite{MR2156694} is for  
partial $t$-intersection sets of partitions with $k \leq  \ell(t-1)$. It is possible that the approach in this paper could be 
applied for larger values of $t$, but there are some steps that we predict will be complicated.

It is straight-forward to generalize the definition of $X_{k,\ell}$ to partially $t$-intersecting partitions by defining the
graph $X_{t, k,\ell}$. This graph will also have $\mathcal{U}_{k, \ell}$ as its vertex set, and two partitions $P$ and $Q$ 
are adjacent if and only if for any pair of blocks $P_i \in P$ and $Q_j \in Q$ we have $|P_i \cap Q_j | <t$. 
A partially $t$-intersecting set of partitions is a coclique in $X_{t, k,\ell}$.  

The conjecture is if $k < \ell(t-1)$, then the maximum cocliques in $X_{t, k,\ell}$ are exactly the 
canonical partially $t$-intersecting sets. The Young subgroup $\sym{[k\ell-t,t]}$ is the stabilizer of a canonically partially $t$-intersecting set. The most significant complication is that for $t>2$, there are 
 more than two irreducible representations in both 
\begin{equation}\label{eq:commonreps}
\indg{1_{\sym{[k\ell-t,t]}}  }{ \sym{k\ell} },  \qquad \indg{1_{\sym{k} \wr \sym{\ell}}}{ \sym{k\ell} }. 
\end{equation}
For this approach given in this paper to work, we believe the eigenvalues belonging to all the irreducible 
representations common to these two induced representations, except the trivial representation, should be the least eigenvalue of $X_{t, k, \ell}$. To make this happen we suspect that a weighted adjacency matrix of $X_{t, k, \ell}$ would be needed in the ratio bound, rather than just the adjacency matrix; the weighting would have to be chosen so that the common modules (except the trivial) in Equation~(\ref{eq:commonreps}) all belong to the same eigenvalue. Another complication is that potentially more of the eigenvalues of $X_{t, k, \ell}$ would have to be calculated, at the very least all the eigenvalues belonging to the common representations would need to be known.

 Bender's theorem is much more general than the version we stated here.  We only state Bender's theorem for matrices with 01-entries, but the full theorem applies to matrices with entries less than $t$. Using the full theorem we would be able approximate the degree of $X_{t, k,\ell}$ for $t\geq 2$.

\bibliographystyle{plain}

\end{document}